\documentclass[12pt]{article}
\usepackage{amsmath,amssymb,amsfonts,amsthm}
\newcounter{item}[section]
\newcounter{kirshr}
\newcounter{kirsha}
\newcounter{kirshb}
\newenvironment{enumarab}{\setcounter{kirshb}{1}
\begin{list}{(\arabic{kirshb})}{\usecounter{kirshb}} }{\end{list}}
\newtheorem{theorem}{Theorem}[section]

\newtheorem{lemma}[theorem]{Lemma}

\newenvironment{demo}[1]{\noindent{\bf #1.}\upshape\mdseries}
{\nopagebreak{\hfill\rule{2mm}{2mm}\nopagebreak}\par\normalfont}
\theoremstyle{definition}

\newtheorem{example}[theorem]{Example}
\newtheorem{definition}[theorem]{Definition}

\def\Q{\mathbb{Q}}
\def\C{{\mathfrak{C}}}
\def\Fm{{\mathfrak{Fm}}}

\def\Nr{{\mathfrak{Nr}}}
\def\Fr{{\mathfrak{Fr}}}
\def\Sg{{\mathfrak{Sg}}}
\def\Fm{{\mathfrak{Fm}}}
\def\A{{\mathfrak{A}}}
\def\B{{\mathfrak{B}}}
\def\C{{\mathfrak{C}}}

\def\CA{{\bf CA}}

\def\(R)RA{{\bf (R)RA}}

\def\Q{\mathbb{Q}}

 \def\CA{{\sf CA}}
\def\B{{\sf B}}

\def\Sg{{\mathfrak{Sg}}}
\def\Nr{{\mathfrak{Nr}}}

\def\cyl#1{{\sf c}_{#1}}

\def\Nr{{\mathfrak{Nr}}}

\def\A{{\mathfrak{A}}}
\def\B{{\mathfrak{B}}}
\def\C{{\mathfrak{C}}}

\def\A{{\mathfrak{A}}}
\def\B{{\mathfrak{B}}}
\def\C{{\mathfrak{C}}}

\def\CA{{\bf CA}}

\title{Completeness, interpolation and omitting types in infinitary predicate topological logic}
\author{Tarek Sayed Ahmed \\
Department of Mathematics, Faculty of Science,\\ 
Cairo University, Giza, Egypt.
  }
\begin{document}
\maketitle

\begin{abstract}
\noindent  We prove all properties in the title for infinitary  topological logics,  using topological infinite dimensional dimension complemented 
cylindric algebras.
%by studying expansions by quantifiers and substitutions of various 
\end{abstract}

\section{Inroduction and basics}
Topological logic was introduced by Makowsky and Ziegler \cite{z} and Sgro \cite{s}.
 Such logics have a classical semantics with a topological touch; and their study was pursued  using algebraic logic by Georgescu\cite{g}. 
The models carry a topology. Here we carry out a similar investigation, using also an algebraic approach, but proving much stronger results.

Throughout, $\alpha$ is an infinite ordinal. 
Instead of taking ordinary set algebras with units of the form $^{\alpha}U$, one 
requires that the base $U$ is endowed with some topology. Then, given such an algebra, 
for each $k<\alpha$ one defines an {\it interior operator} on $\wp(^{\alpha}U)$ by
$$I_k(X)=\{s\in {}^{\alpha}U; s_k\in int\{s\in U: s_u^k\in X\}, X\subseteq {}^{\alpha}U.$$
Notice that in the case of discrete topolgy this gives nothing new, we are just adding identity operators, which adds absolutely nothing at all.

These operations can also be defined on weak spaces, that is sets of sequences agreeing cofinitely with a given fixed sequence.
In more detil, a weak space is one of the form $\{s\in ^{\alpha}U \{i\in \alpha: s_i\neq p_i\}|<\omega\}$, for a given fixed in advance $p\in {}^{\alpha}U$.
such a space is denoted by $^{\alpha}U^{(p)}$. 

Now such algebras lend itself to an abstract formulation aiming to capture he concrete set algebras; or rather the variety generted by them.
This consists of expanding the similarity types unary operators, 
one for each $k<\alpha$, satisfying certain identities; we abbreviate these algebras by $TCA_{\alpha}$, 
short for topological cylindric  algebras.

Such an axiomatization is essentially equivalent to that given in \cite{g}, 
since locally finite polyadic algebras and cylindric algebras are  equivalent.

In \cite{g}  a representation theorem is proved  locally finite algebras; 
here we extend this theorem in 3 ways. We prove a strong representation theorem for dimensiion complemented algebras,
the logic corresponding to such algebras and interpolation for dimension complemented algebras, 
and we prove omitting types and interpolation for such logics. 
The constructions used are standard Heknin constructions; for the semantics given allows such 
proofs. Finally we count the pairwise-non isomorphic models, and we show that in non trivial cases if $T$ is a countable theory then the number of 
pairwise non-isomorphic models is the continuum.
We start by the definition of cylindric algebras:
\begin{definition}
Let $\alpha$ be an ordinal. A cylindric algebra of dimension $\alpha$, a  $\CA_{\alpha}$ for short, 
is defined to be an algebra
$$\C=\langle C, +, -, 0, 1, \cyl{i}, \sf{d_{ij}}  \rangle_{i,j\in \alpha}$$
obeying the following axioms for every $x,y\in C$, $i,j,k<\alpha$

\begin{enumerate}

\item The equations defining boolean algebras

\item $\cyl{i}0=0$

\item $x\leq \cyl{i}x$

\item $\cyl{i}(x\cdot \cyl{i}y)=\cyl{i}x\cdot .\cyl{i}y$

\item $\cyl{i}\cyl{j}x=\cyl{j}\cyl{i}x$

\item ${\sf d}_{ii}=1$

\item if $k\neq i,j$ then ${\sf d}_{ij}=\cyl{k}({\sf d}_{ik}.{\sf d}_{jk})$

\item If $i\neq j$, then ${\sf c}_i({\sf d}_{ij}.x).{\sf c}_i({\sf d}_{ij}.-x)=0$

\end{enumerate}
\end{definition}
For a cylindric algebra $\A$, we set $c_i^{\partial}x=-c_i-x$ and $s_i^j(x)=c_i(d_{ij}.x)$.
We consider only infinite dimensional cylindric algebras.
In that we follow the notation of \cite{HMT1}. We are concerned with  algebras of dimension $\alpha$ endowed by unary operation $I(i)$ 
for each $i\in \alpha$
\begin{definition} An interior cylindric algebra  is of the form $(\A,I(i))_{i<\alpha}$ where $\A\in CA_{\alpha}$ and for each $i<\alpha$, $I(i)$ is a unary operation on 
$A$ called an interior operators satisfying:
\begin{enumerate}
\item $c_i^{\partial}(p\leftrightarrow q)\leq \forall i (I(i)p\leftrightarrow I(i)q)$

Here $p\leftrightarrow q$ is short for $[(-p+q).(-q+p).]$ 
\item $I(i)p\leq p$
\item $I(i)p\cdot I(i)p=I(i)(p\cdot q)$
\item $p\leq I(i)I(i)p$
\item $I(i)1=1$
\item $s_j^iI(i)p=I(j)s_j^ip$
\end{enumerate}
\end{definition}

%The typical example of an interior  cylindric algebra is taking a $Cs_{\alpha}$ with base $U$ 
%where $U$ has a topology and one defines $I(i)$ by
%$$I_i(X)=\{s\in {}^{\alpha}U: u_k\in int\{a\in U: u^k_a\in X\}\}.$$ In our work however, we shall consider representations via 
%weak set algebras $Ws_{\alpha}.$
%$A\in Ws_{\alpha}$ if its greatest element is of the form $^{\alpha}U^{p)}=\{s\in {}^{\alpha}U: |\{i\in \alpha: s_i\neq p_ i\}|<\omega\}$. 
%It is easy to see that such algebras endowed with the above operations is an interior algebra.

\section{Completeness and Interpolation}
To prove completeness we formulate and prove two lemmas. Properties of substitutions reported in \cite{HMT1} 
are freely used. For example, for every finite transformation $\tau$ we have a unary operation
${\sf s}_{\tau}$ tha happens to be a Boolean endomorphism.

\begin{lemma} Let $\C$ be a cylindric algebra. 
Let $F$ be a Boolean  ultrafilter in $C$. Define the relation $E$ on $\omega$ by $(i,j)\in E$ if and only if
$d_{ij}\in F$. Then $E$ is an equivalence relation.
\end{lemma}
 \begin{proof}
$E$ is reflexive because ${\sf d}_{ii}=1$ and symmetric 
because ${\sf d}_{ij}={\sf d}_{ji}.$
$E$ is transitive because   $F$ is a filter and for all $k,l,u<\alpha$, with $l\notin \{k,u\}$, 
we have 
${\sf d}_{kl}.{\sf d}_{lu}\leq {\sf c}_l({\sf d}_{kl})$
\end{proof}

\begin{lemma} \label{t} Let everything be as in above lemma, but assume that $\C$ is dimension complemented. Let
$V=\{\tau\in {}^{\alpha}\alpha: |\{i\in \alpha: \tau(i)\neq i\}|<\omega\}$.
Let $F$ be 
the Boolean ultrafilter of $C$, and $E$ the equivalence relation on $\omega$. 
For $\sigma, \tau\in V$,  write 
$$\sigma\equiv_E\tau\textrm {  iff } 
(\forall i\in \mu) (\sigma (i),\tau(i))\in E.$$ 
and let 
$$\bar{E}=\{(\sigma,\tau)\in {}^2V: \sigma\equiv_E \tau\}.$$
Then $\bar{E}$ is an euivalence relation on $V$. Let $W= V/\bar{E}$ 
For $h\in W,$ write $h=\bar{\tau}$ for $\tau\in V$ such that
$\tau(j)/E=h(j)$ for all $j\in \mu$. 
Let 
$f(x)=\{ \bar{\tau} \in W: {\sf s}_{\tau}x\in F\}.$  
Then $f$ is well defined and is a homomorphism.
\end{lemma}
\begin{proof} For this, it clearly suffices to show that for  
$\sigma, \tau\in V$ and $x\in {\cal A}_i$ if $\sigma \bar{E} \tau$, 
then $${\sf s}_{\tau}x\in F\textrm { iff } {\sf s}_{\sigma}x\in F.$$
This can be proved by induction on the cardinality of 
$$J=\{i\in \mu: \sigma i\neq \tau i\}.$$
Then it is easy to check that $f$ is a homomrphism follows from \cite{IGPL}.
\end{proof}

\begin{theorem} 
Let $(\A, I)$ be an interior algebra such that $\A$ is dimension complemented. 
Then for every non-zero $a$ there is a $(\B, I)$ such that $\B$ is weak set algebra, with an interior operator,
anf $f:\A\to \B$ such that $f(a)\neq 0.$
\end{theorem}
\begin{demo}{Proof} Although the proof can be proved more directly without  using dilations of $\A$, 
we prefer to resort to a neat embedding theorem to prepare for the next proof

Assume that $\A=\Nr_{\alpha}\B$ where $\B$ is dimension complemented and $\B\in Dc_{\kappa}$, $\kappa$ a regular cardinal.
Such a $\B$ exists exactly like the case and the interior opeartions are induced the natural way.
Arrange $\kappa\times \B$ 
into $\kappa$-termed sequences:
$\langle (k_i,x_i): i\in \kappa \rangle$
Since $\kappa$ is regular, we can define by recursion (or step-by-step) 
$\omega$-termed sequences of witnesses: 
$\langle u_i:i\in \kappa\rangle$ 
such that for all $i\in \kappa$ we have:
$$u_i\in \mu\smallsetminus
\Delta a\cup \cup_{j\leq i}(\Delta x_j\cup \Delta y_j)\cup \{u_j:j<i\}\cup \{v_j:j<i\}.$$
The regularity of $\kappa$ guarntees this.
Let  $$Y_1= \{a\}\cup \{-{\sf  c}_{k_i}x_i+{\sf s}_{u_i}^{k_i}x_i: i\in \kappa\},$$
Then the filter generated by $Y_1$ is proper. Extend to to an ultafilter. Define the equivalence relation as above,
and 
then define  $f:\A\to \wp(V)$ via
$$x\mapsto \{\bar{\tau} \in V: s_{\tau}x\in F\}$$
where $V$ is the set $^{\alpha}\alpha^{(Id)}$.
Let $$q=\{\{\bar{k}\in \beta: s_k^iI(i)p\in F\}: p\in A, i\in \alpha\}.$$
To define the interior operations, we set
for each $i<\kappa$ 
$$J(i): \wp(V)\to \wp (V)$$
by $$\bar{x}\in J(i)A\Longleftrightarrow \exists U\in q (x_i\in U\subseteq \{u\in \beta: x^i_u\in A\}.)$$ 
Here $\bar{x}$ is class of $x$ induced by the above relation. Then it is not hard to check that
$$\bar{x}\in \Psi(I(i)p)\Longleftrightarrow s_{x}I(i)p\in F\Longleftrightarrow s^i_{x_i}I(i) s^{j_1,\ldots j_n}_{x_1,\ldots x_n}p\in F.$$
We let $$y=[j_1|x_1]\ldots [j_n|x_n].$$ 
Now,
$$x\in J(i)\Psi(p)\Longleftrightarrow \exists U\in q (x_i\in U\subseteq \{u: x^i_u\in \psi(p)\}.$$
Every cylindric operation is preserved, it remains to check the interior operations.
The proof is very similar to that in \cite{g}. 
%'
We need to show: 
$$\psi(I(i)p)=J(i)\psi(p).$$
Since $x\in \psi(I(i)$, we have $x_i\in \{u: {\sf s}_u^i I(i){\sf s}_yp\in F\}\in q$
But $I(i){\sf s}_yp\leq {\sf s}_y p$ hence
$U=\{u: {\sf s}_i^u I(i){\sf s}_yp\in F\}\subseteq \{u: {\sf s}_i^u{\sf s}_yp\in F\}.$
Since $y=[j_1|x_1]\ldots [j_n|x_n],$ then we have:
$x_i\in U\subseteq \{u: x^i_u\in \Psi(p)\}.$
Thus $x\in J(i)\psi(p).$

Now we prove the converse direction, which is slightly harder.
Let $x\in J(i)\Psi(p)$. Let $U\in q$ such that 
$$x_i\in U\subseteq \{u: {\sf s}_u^i{\sf s}_xp\in F\},$$
%$$U=\{ u: s_u^j I(j)r\}\in F.$$
For every $u\in \beta,$ $u$ large enough, we have:
$${\sf s}_u^jI(j)r\in F\Longleftrightarrow {\sf s}_u^j{\sf s}_xp\in F,$$
 $${\sf s}_u^jI(j)r\land {\sf s}_u^i{\sf s}_xp\in F\Longleftrightarrow {\sf s}_u^jI(j)r\in F.$$
But
$${\sf s}_u^jI(j)r={\sf s}_u^iI(i){\sf S}(i|j)r,$$
So, containing the chain of equivalence
$${\sf s_u}^iI(i)s_j^ir\land {\sf s}_xp\leftrightarrow I(i){\sf S}(i|j)r\in F$$
$$ \forall i (I(i)S(i|j)r\land I(i) {\sf s}_xp\leftrightarrow I(i){\sf S}(i|j)r\in F$$
$${\sf s}_u^iI(i)S(i|j)r)\land {\sf s}_u^iI(i){\sf s}_xp\leftrightarrow {\sf s}_u^iI(i)S(i|j)r\in F$$
$${\sf s}_u^jI(j)r\land {\sf s}_u^jI(i)s_xp\leftrightarrow {\sf s}_u^jI(j)r\in F.$$
\end{demo}

\begin{theorem}Let $\alpha$ be an infinite ordinal. let $\beta$ be a cardinal.  Let $\rho:\beta\to \wp(\alpha)$ such that
$\alpha\sim \rho(i)$ is infinite for all $i\in \beta$. Then $\Fr_{\beta}^{\rho}TCA_{\alpha}$ has the interpolation property.
\end{theorem}
\begin{proof}
Let $\A=\Fr_{\beta}^{\rho}TCA_{\alpha}$. Let $a\in \in \Sg X_1$ and $c\in \Sg X_2$ be such that $a\leq c$. We want to find an interpolant in 
$\Sg^{\A}(X_1\cap X_2)$. Let $\B\in TCA_{\kappa}$, $\kappa$ a regular cardinal, such that $\A=\Nr_{\alpha}\B$. 
Assume that no such interpolant exists in $\A$, then no interpolant exists in $\B$, because if $b$ is an interpolant
in $\Sg^{\B}(X_1\cap X_2),$ then there exists a finite set $\Gamma\subseteq \kappa\sim \alpha$, such that
${\sf c}_{(\Gamma)}b\in \Nr_{\alpha}\Sg^{\B}(X_1\cap X_2)=\Sg^{\Nr_{\alpha}\B}(X_1\cap X_2)=\Sg^{\A}(X_1\cap X_2)$; which is clearly an interpolant in $\A$.

Arrange $\kappa\times \Sg^{\cal B}(X_1)$ 
and $\kappa\times \Sg^{\cal B}(X_2)$ 
into $\kappa$-termed sequences:
$$\langle (k_i,x_i): i\in \kappa\rangle\text {  and  }\langle (l_i,y_i):i\in \kappa\rangle
\text {  respectively.}$$ Since $\kappa$ is regular, we can define by recursion 
$\omega$-termed sequences of witnesses: 
$$\langle u_i:i\in \kappa\rangle \text { and }\langle v_i:i\in \kappa\rangle$$ 
such that for all $i\in \kappa$ we have:
$$u_i\in \mu\smallsetminus
(\Delta a\cup \Delta c)\cup \cup_{j\leq i}(\Delta x_j\cup \Delta y_j)\cup \{u_j:j<i\}\cup \{v_j:j<i\}$$
and
$$v_i\in \mu\smallsetminus(\Delta a\cup \Delta c)\cup 
\cup_{j\leq i}(\Delta x_j\cup \Delta y_j)\cup \{u_j:j\leq i\}\cup \{v_j:j<i\}.$$

For a boolean algebra $\cal C$  and $Y\subseteq \cal C$, we write 
$fl^{\cal C}Y$ to denote the boolean filter generated by $Y$ in $\cal C.$  Now let 
$$Y_1= \{a\}\cup \{-{\sf  c}_{k_i}x_i+{\sf s}_{u_i}^{k_i}x_i: i\in \kappa\},$$
$$Y_2=\{-c\}\cup \{-{\sf  c}_{l_i}y_i +{\sf s}_{v_i}^{l_i}y_i:i\in \kappa\},$$
$$H_1= fl^{Bl\Sg^{B}(X_1)}Y_1,\  H_2=fl^{Bl\Sg^B(X_2)}Y_2,$$ and 
$$H=fl^{Bl\Sg^{B}(X_1\cap X_2)}[(H_1\cap \Sg^{B}(X_1\cap X_2)
\cup (H_2\cap \Sg^{B}(X_1\cap X_2)].$$
We claim that $H$ is a proper filter of $\Sg^{\B}(X_1\cap X_2).$
To prove this it is sufficient to consider any pair of finite, strictly
increasing sequences of natural numbers 
$$\eta(0)<\eta(1)\cdots <\eta(n-1)<\omega\text { and } \xi(0)<\xi(1)<\cdots 
<\xi(m-1)<\omega,$$
and to prove that the following condition holds:

(1) For any $b_0$, $b_1\in \Sg^{B}(X_1\cap X_2)$ such that
$$a\odot.\prod_{i<n}(-{\sf  c}_{k_{\eta(i)}}x_{\eta(i)}+{\sf s}_{u_{\eta(i)}}^{k_{\eta(i)}}x_{\eta(i)})\leq b_0$$  
and
$$(-c)\odot\prod_{i<m}
(-{\sf  c}_{l_{\xi(i)}}y_{\xi(i)}+{\sf s}_{v_{\xi(i)}}^{l_{\xi(i)}}y_{\xi(i)})\leq b_1$$
we have 
$$b_0+b_1\neq 0.$$
We prove this by induction on $n+m$

Proving that $H$ is a proper filter of $\Sg^{\cal B}(X_1\cap X_2)$,
let $H^*$ be a (proper boolean) ultrafilter of $\Sg^{\cal B}(X_1\cap X_2)$
containing $H.$
We obtain  ultrafilters $F_1$ and $F_2$ of $\Sg^{\cal B}(X_1)$ and $\Sg^{\cal B}(X_2)$,
respectively, such that
$$H^*\subseteq F_1,\ \  H^*\subseteq F_2$$
and (**)
$$F_1\cap \Sg^{\cal B}(X_1\cap X_2)= H^*= F_2\cap \Sg^{\cal B}(X_1\cap X_2).$$
Now for all $x\in \Sg^{\cal B}(X_1\cap X_2)$ we have
$$x\in F_1\text { if and only if } x\in F_2.$$
Also from how we defined our ultrafilters, $F_i$ for $i\in \{1,2\}$ satisfy the following
condition:

(*) For all $k<\mu$, for all $x\in Sg^{\cal B}X_i$
if ${\sf  c}_kx\in F_i$ then ${\sf s}_l^kx$ is in $F_i$ for some $l\notin \Delta x.$
We obtain  ultrafilters $F_1$ and $F_2$ of $\Sg^{\B}X_1$ and $\Sg^{\B}X_2$, 
respectively, such that 
$$H^*\subseteq F_1,\ \  H^*\subseteq F_2$$
and (**)
$$F_1\cap \Sg^{\B}(X_1\cap X_2)= H^*= F_2\cap \Sg^{\B}(X_1\cap X_2).$$
Now for all $x\in \Sg^{\cal B}(X_1\cap X_2)$ we have 
$$x\in F_1\text { if and only if } x\in F_2.$$ 
%Also from how we defined our ultrafilters, $F_i$ for $i\in \{1,2\}$ are perfect.  

Then define the homomorphisms, one on each subalgebra, exactly like above 
then freeness will enable pase these homomophisms, to a single one defined to the set of free generators, 
which we can assume to be, without any loss, to 
be $X_1\cap X_2$ and it will satisfy  $h(a.-c)\neq 0$ which is a contradiction.

\end{proof}

\section{Omitting types, and counting models}

\begin{theorem}\label{infinite} 
\begin{enumarab}
\item Let $\A\in TA_n$ be  countable. Assume that 
$\kappa<covK$, where $covK$ is the least cardinal $\kappa$ such that the real line can be covered by
pairwise disjoint $\kappa$ nowhere-dense sets. Let $(\Gamma_i:i\in \kappa)$ be a set of non-principal types in $\A$. 
Then there is a interior weak set algebra $\B$ with an interior operator,
and an injective homomorphism  $f:\A\to \B$ such that $\bigcap_{x\in X_i}f(x)=\emptyset$, and $f(a)\neq 0.$
\item If the $X_i$ are ultrafilters, then we can omit $\lambda< 2^{\omega}.$
\end{enumarab}
\end{theorem}
\begin{demo}{Proof} 
For the first part, we have by \cite[1.11.6]{HMT1} that 
\begin{equation}\label{t1}
\begin{split} (\forall j<\alpha)(\forall x\in A)({\sf c}_jx=\sum_{i\in \alpha\smallsetminus \Delta x}
{\sf s}_i^jx.)
\end{split}
\end{equation}
Now let $V$ be the weak space $^{\omega}\omega^{(Id)}=\{s\in {}^{\omega}\omega: |\{i\in \omega: s_i\neq i\}|<\omega\}$.
For each $\tau\in V$ for each $i\in \kappa$, let
$$X_{i,\tau}=\{{\sf s}_{\tau}x: x\in X_i\}.$$
Here ${\sf s}_{\tau}$ 
is the unary operation as defined in  \cite[1.11.9]{HMT1}.
For each $\tau\in V,$ ${\sf s}_{\tau}$ is a complete
boolean endomorphism on $\B$ by \cite[1.11.12(iii)]{HMT1}. 
It thus follows that 
\begin{equation}\label{t2}\begin{split}
(\forall\tau\in V)(\forall  i\in \kappa)\prod{}^{\A}X_{i,\tau}=0
\end{split}
\end{equation}
Let $S$ be the Stone space of the Boolean part of $\A$, and for $x\in \A$, let $N_x$ 
denote the clopen set consisting of all
boolean ultrafilters that contain $x$.
Then from \ref{t1}, \ref{t2}, it follows that for $x\in \A,$ $j<\beta$, $i<\kappa$ and 
$\tau\in V$, the sets 
$$\bold G_{j,x}=N_{{\sf c}_jx}\setminus \bigcup_{i\notin \Delta x} N_{{\sf s}_i^jx}
\text { and } \bold H_{i,\tau}=\bigcap_{x\in X_i} N_{{\sf s}_{\bar{\tau}}x}$$
are closed nowhere dense sets in $S$.
Also each $\bold H_{i,\tau}$ is closed and nowhere 
dense.
Let $$\bold G=\bigcup_{j\in \beta}\bigcup_{x\in B}\bold G_{j,x}
\text { and }\bold H=\bigcup_{i\in \kappa}\bigcup_{\tau\in V}\bold H_{i,\tau.}$$
By properties of $covK$, it can be shown $\bold H$ is a countable collection of nowhere dense sets.
By the Baire Category theorem  for compact Hausdorff spaces, we get that $H(A)=S\sim \bold H\cup \bold G$ is dense in $S$.
Accordingly let $F$ be an ultrafilter in $N_a\cap X$.
By the very choice of $F$, it follows that $a\in F$ and  we have the following 
\begin{equation}
\begin{split}
(\forall j<\beta)(\forall x\in B)({\sf c}_jx\in F\implies
(\exists j\notin \Delta x){\sf s}_j^ix\in F.)
\end{split}
\end{equation}
and 
\begin{equation}
\begin{split}
(\forall i<\kappa)(\forall \tau\in V)(\exists x\in X_i){\sf s}_{\tau}x\notin F. 
\end{split}
\end{equation}
Next we form the canonical representation corresponding to $F$
in which satisfaction coincides with genericity. 
To handle equality. we define
$$E=\{(i,j)\in {}^2{\alpha}: {\sf d}_{ij}\in F\}.$$
$E$ is an equivalence relation on $\alpha$.   
$E$ is reflexive because ${\sf d}_{ii}=1$ and symmetric 
because ${\sf d}_{ij}={\sf d}_{ji}.$
$E$ is transitive because $F$ is a filter and for all $k,l,u<\alpha$, with $l\notin \{k,u\}$, 
we have 
$${\sf d}_{kl}\cdot {\sf d}_{lu}\leq {\sf c}_l({\sf d}_{kl}\cdot {\sf d}_{lu})={\sf d}_{ku}.$$
Let $M= \alpha/E$ and for $i\in \omega$, let $q(i)=i/E$. 
Let $W$ be the weak space $^{\alpha}M^{(q)}.$
For $h\in W,$ we write $h=\bar{\tau}$ if $\tau\in V$ is such that
$\tau(i)/E=h(i)$ for all $i\in \omega$. $\tau$ of course may
not be unique.
Define $f$ from $\B$ to the full weak set algebra with unit $W$ as follows:
$$f(x)=\{ \tau\in {}^{\omega}\omega:  {\sf s}_{\tau}x\in F\}, \text { for } x\in \A.$$ 
Then it can be checked that $f$
is a homomorphism 
such that $f(a)\neq 0$ and 
$\bigcap f(X_i)=\emptyset$ for all $i\in \kappa$, hence the desired conclusion.

For the second part, the idea is that one can build several models such that they overlap only on isolated types.
One can build {\it two} models so that every maximal type which is realized in both is isolated. Using the jargon of Robinson's finite forcing
implemented via games, the idea
is that one distributes this job of building the two models among experts, each has a role to play, and that all have 
winning strategies. There is no difficulty in stretching the above idea to make the experts build three, four or any finite number of models 
which overlap only at principal types. 
With a pinch of diagonalisation we can extend the number to $\omega$. 

To push it still further to $^{\omega}2$ needs an entirely new idea (due to Shelah), which we will implement.
%This is one of those many places in model theory where we get continuum many models for the same price as two.

Algebraically, we first construct two representations of $\B$ such that if $F$ is an ultrafilter in $B$ that is realized in both representations, then $F$ is
necessarily principal, that is $\prod F$ is an atom generating $F$, then we sketch the idea of how to obtain $^{\omega}2$ many.
We construct two ultrafilters $T$ and $S$ of $\B$ such that (*)
$\forall \tau_1, \tau_2\in {}^{\omega}\omega( G_1=\{a\in \B: {\sf s}_{\tau_1}a\in T\},
G_2=\{a\in \B: s_{\tau_1}a\in S\})\\\implies G_1\neq G_2 \text { or $G_1$ is principal.}$
Note that $G_1$ and $G_2$ are indeed ultrafilters.
We construct $S$ and $T$ as a union of a chain. We carry out various tasks as we build the chains.
The tasks are as in (*), as well as

(**) for all $a\in A$, if $c_ka\in T$, then $s_l^kx\in T$ for $l\notin \Delta x$.

(***) for all $a\in A$ either $a\in T$ or $-a\in T$, and same for $S$.

We let $S_0=T_0=\{1\}$.
There are countably many tasks. Metaphorically we hire countably many experts and give them one task each.
We partition $\omega$ into infinitely many sets and we assign one of these tasks to each expert.
When $T_{i-1}$ and $S_{i-1}$ have been chosen and $i$ is in the set assigned to some expert $E$, then $E$ will construct
$T_i$ and $S_i$.

Let us start with  task $(**)$. The expert waits until she is given a set $T_{i-1}$ which contains $c_ka$ for some $k<\omega$.
 Every time this happen she look for a {\it witness} $l$ which is outside elements in $T_{i-1}$; 
this is possible since the latter is finite, then she sets $T_i=T_{i-1}\cup \{s_k^la\}$. 
Otherwise, she does nothing. This strategy works because her subset
of $\omega$ is infinite, hence contains arbitrarily large numbers. Same for $S_i$.

Now consider the expert who handles task (**). Let $X$ be her subset of $\omega$. Let her list as $(a_i: i\in X)$ all elements of $X$.
When $T_{i-1}$ has been chosen with $i\in X$, she should consider whether $T_{i-1}\cup \{a_i\}$ is consistent. If it is she puts
$T_i=T_{i-1}\cup \{a_i\}$. If not she puts $T_i=T_{i-1}\cup \{-a_i\}$. Same for $S_i$.
%Next consider the expert who deals with the tasks in \ref{t5}. She waits until she is gets a set $T_{i-1}$ which contains ${\sf c}_ka$.
%Every time this happens she chooses $l\notin \Delta a$
%which is not used in $T_{i-1}$, and she puts $T_i=T_{i-1}\cup \{{\sf s}_k^la\}$. Same for $S_i$.

Now finally consider the crucial tasks in (*). Suppose that $X$ contains $i,$ and $S_{i-1}$ and $T_{i-1}$ have been chosen.
Let $e=\bigwedge S_{i-1}$ and $f=\bigwedge T_{i-1}$. We have two cases.
If $e$ is an atom in $B$ then the ultrafilter $F$ containing $e$ is principal so our expert can put $S_i=S_{i-1}$ and $T_i=T_{i-1}$.
If not, then let $F_1$ , $F_2$ be distinct ultrafilters containing $e$. Let $G$ be an ultrafilter containing $e$, and assume that $F_1$ 
is different from $G$.
Let $\theta$ be in $F_1-G$. Then put $S_i=S_{i-1}\cup \{\theta\}$ and $T_i=T_{i-1}\cup \{-\theta\}.$
It is not hard to check that the canonical models, defined the usual way,  corresponding to $S$ and $T$ are as required.

To extend the idea, we allow experts at any stage to introduce a new chain of theories which is a duplicate copy of one of the chains being
constructed. The construction takes the form of a tree where each branch separately will give a chain of conditions.
By splitting the tree often enough the experts can guarantee that there are continuum many branches and hence continuum many
representations. This is a well know method, in model theory, when one gets $^{\omega}2$ many models for the price of two.
There is one expert whose job is to make sure that this property 
is enforcable for each pair of branches.
But she can do this task, because at each step the number of branches is still finite.

Assume not. Let ${\bold F}$ be the given set of non principal ultrafilters. Then for all $i<{}^{\omega}2$,
there exists $F$ such that $F$ is realized in $\B_i$. Let $\psi:{}^{\omega}2\to \wp(\bold F)$, be defined by
$\psi(i)=\{F: F \text { is realized in  }\B_i\}$.  Then for all $i<{}^{\omega}2$, $\psi(i)\neq \emptyset$.
Furthermore, for $i\neq j$, $\psi(i)\cap \psi(j)=\emptyset,$ for if $F\in \psi(i)\cap \psi(j)$ then it will be realized in
$\B_i$ and $\B_j$, and so it will be principal.  This implies that $|\bold F|={}^{\omega}2$ which is impossible.

%\end{enumroman}
%\end{proof}
%\end{proof}

\end{demo}

%And then proceed s above.
Now we count the non-isomorphic models, but first some definitions:
\begin{definition}
Let $\A$ and $\B$ be set algebras with bases $U$ and $W$ respectively. Then $\A$ and $\B$
are \emph{base isomorphic} if there exists a bijection
$f:U\to W$ such that $\bar{f}:\A\to \B$ defined by ${\bar f}(X)=\{y\in {}^{\alpha}W: f^{-1}\circ y\in x\}$ is an isomorphism from $\A$ to $\B$
\end{definition}
\begin{definition} An algebra $\A$ is \emph{hereditary atomic}, if each of its subalgebras is atomic.
\end{definition}
Finite Boolean algebras are hereditary atomic of course,
but there are infinite hereditary atomic Boolean algebras; any Boolean algebra generated by by its atoms is
hereditary atomic, for example the finite co-finite
algebra on any set. An algebra that is infinite and complete is not hereditory atomic, wheter atomic or not.

\begin{example}
Hereditary atomic algebras arise naturally as the Tarski Lindenbaum algebras of
certain countable first order theories, that abound. If $T$ is a countable complete first order theory
which has an an $\omega$-saturated model, then for each $n\in \omega$,
the Tarski Lindenbuam Boolean algebra $\Fm_n/T$ is hereditary atomic. Here $\Fm_n$ is the set of formulas using only
$n$ variables. For example $Th(\Q,<)$ is such with $\Q$ the $\omega$ saturated model.
\end{example}

A well known model-theoretic result is that $T$ has an $\omega$ saturated model iff $T$ has countably many $n$ types
for all $n$. Algebraically $n$ types are just ultrafilters in $\Fm_n/T$.
And indeed, what characterizes hereditary atomic algebras is that the base of their Stone space, that is the set of all
ultrafilters, is at most countable.

\begin{lemma}\label{b} Let $\B$ be a countable  Boolean algebra. If $\B$ is hereditary atomic then the number of ultrafilters is at most countable; ofcourse they are finite
if $\B$ is finite. If $\B$ is not hereditary atomic the it has $2^{\omega}$ ultarfilters.
\end{lemma}
\begin{proof}\cite{HMT1} p. 364-365  for a detailed discussion.
\end{proof}
Our next theorem is the, we believe, natural extension of Vaught's theorem to variable rich languages.
However, we address only languages with finitely many relation symbols. (Our algebras are finitely generated,
and being simple, this is equivalent to that it is
generated by a single element.)
%A famous conjecture of Vaught says that the number of non-isomorphic countable models of a complete theory
%is either $\leq \omega$ or exactly $^{\omega}2$. We show that this is the case for the multi (infinite) dimensional modal logic corresponding
%to $SA_{\alpha}$.

Now let us see  how far we can get, with proving an analogue of counting distinguishable models. We now count
distinguishable {\it weak} models.
Let $\A\in Dc_{\alpha}$. Now we hav only {\it finite} substitutions. As before, let
$$\mathcal{H}(\A)=\bigcap_{i<\omega,x\in A}(N_{-c_ix}\cup\bigcup_{j<\omega}N_{s^i_jx})$$ and, in the cylindric algebraic case, let
and
$$\mathcal{H}'(\A)=\mathcal{H}(\A)\cap\bigcap_{i\neq j\in\omega}N_{-d_{ij}}.$$

Now $\mathcal{H}(\A)$ and $\mathcal{H}'(\A)$ are $G_\delta$ subsets of $\A^*$, and are nonempty, in fact they are dense, and they
are  Polish spaces; 
Assume $\mathcal F\in \mathcal{H}(\A).$ For any $x\in A$, define the function
$\mathrm{rep}_{\mathcal F}$ to be
$$\mathrm{rep}_{\mathcal F}(x)=\{\tau\in{}^\omega\omega^{Id}:s_\tau x\in \mathcal F\}.$$
%We have the following results due to G. S\'agi and D. Szir\'aki; (see \cite{Sagi}).

\begin{theorem}\label{2} Let $\A\in Dc_{\alpha}$ be countable simple and finitely generated.
Then the number of non-base isomorphic representations of $\A$ is $2^{\omega}$.
\end{theorem}
\begin{proof} Let $V={}^{\alpha}\alpha^{(Id)}$ and let $\A$ be as in the hypothesis. Then $\A$ cannot be atomic \cite{HMT1} corollary 2.3.33,
least hereditary atomic. By \ref{b}, it has $2^{\omega}$ ultrafilters.

For an ultrafilter $F$, let $h_F(a)=\{\tau \in V: s_{\tau}a\in F\}$, $a\in \A$.
Then $h_F\neq 0$, indeed $Id\in h_F(a)$ for any $a\in \A$, hence $h_F$ is an injection, by simplicity of $\A$.
Now $h_F:\A\to \wp(V)$; all the $h_F$'s have the same target algebra.
We claim that $h_F(\A)$ is base isomorphic to $h_G(\A)$ iff there exists a finite bijection $\sigma\in V$ such that
$s_{\sigma}F=G$.
We set out to confirm our claim. Let $\sigma:\alpha\to \alpha$ be a finite bijection such that $s_{\sigma}F=G$.
Define $\Psi:h_F(\A)\to \wp(V)$ by $\Psi(X)=\{\tau\in V:\sigma^{-1}\circ \tau\in X\}$. Then, by definition, $\Psi$ is a base isomorphism.
We show that $\Psi(h_F(a))=h_G(a)$ for all $a\in \A$. Let $a\in A$. Let $X=\{\tau\in V: s_{\tau}a\in F\}$.
Let $Z=\Psi(X).$ Then
\begin{equation*}
\begin{split}
&Z=\{\tau\in V: \sigma^{-1}\circ \tau\in X\}\\
&=\{\tau\in V: s_{\sigma^{-1}\circ \tau}(a)\in F\}\\
&=\{\tau\in V: s_{\tau}a\in s_{\sigma}F\}\\
&=\{\tau\in V: s_{\tau}a\in G\}.\\
&=h_G(a)\\
\end{split}
\end{equation*}
Conversely, assume that $\bar{\sigma}$ establishes a base isomorphism between $h_F(\A)$ and $h_G(\A)$.
Then $\bar{\sigma}\circ h_F=h_G$.  We show that if $a\in F$, then $s_{\sigma}a\in G$.
Let $a\in F$, and let $X=h_{F}(a)$.
Then, we have
\begin{equation*}
\begin{split}
&\bar{\sigma\circ h_{F}}(a)=\sigma(X)\\
&=\{y\in V: \sigma^{-1}\circ y\in h_{F}(X)\}\\
&=\{y\in V: s_{\sigma^{-1}\circ y}a\in F\}\\
&=h_G(a)\\
\end{split}
\end{equation*}
Now we have $h_G(a)=\{y\in V: s_{y}a\in G\}.$
But $a\in F$. Hence $\sigma^{-1}\in h_G(a)$ so $s_{\sigma^{-1}}a\in G$, and hence $a\in s_{\sigma}G$.

Define the equivalence relation $\sim $ on the set of ultrafilters by $F\sim G$, if there exists a finite permutation $\sigma$
such that $F=s_{\sigma}G$. Then any equivalence class is countable, and so we have $^{\omega}2$ many orbits, which correspond to
the non base isomorphic representations of $\A$.
\end{proof}
The above theorem is not so  deep, as it might appear on first reading. The relatively simple
proof is an instance of the obvious fact that if a countable Polish group, acts on an uncountable Polish space, then the number of induced orbits
has the cardinality of the continuum, because it factors out an uncountable set by a countable one. In this case, it is quite easy to show
that  the Glimm-Effros Dichotomy holds.

\begin{theorem}
Let $T$ be a countable theory in a rich language, with only finitely many relation symbols,
and $\Gamma =\{\Gamma_i: i\in covK\}$ be non isolated types.
Then $T$ has $2^{\omega}$ weak models that omit $\Gamma$.
\end{theorem}
\begin{proof}One takes $\bold H= H\cap -\bigcup_{i<\lambda}\bigcap_{x\in x_i, \tau \in V} N_{s_{\tau}}x.$
The $< covK$ union can be reduced to a countable union, and hence $\bold H$ is still dense, and the corresponding equivalence relation remains Borel.
\end{proof}

\end{document}